\documentclass[10pt]{article}

\usepackage{color}\usepackage{graphicx}\usepackage{mathrsfs,amsfonts,latexsym,amssymb,amsthm,amsmath}
\usepackage{anysize}\marginsize{2cm}{2cm}{1cm}{2.5cm}
\usepackage[font=footnotesize]{caption}

\newtheorem{theorem}{Theorem}\newtheorem{lemma}{Lemma}\newtheorem{proposition}{Proposition}\newtheorem{corollary}{Corollary}
\theoremstyle{definition}
\newtheorem{criterion}{Criterion}
\renewcommand{\thefootnote}{\fnsymbol{footnote}}

\usepackage[affil-it]{authblk}

\usepackage{algorithm}
\usepackage{algorithmic}
\usepackage{mathabx}

\date{}

\numberwithin{equation}{section}
\numberwithin{theorem}{section}

\begin{document}

\title{On the constrained mock-Chebyshev least-squares}


\author[1]{S. De Marchi}
\author[2]{F. Dell'Accio}
\author[3]{M. Mazza}

\affil[1]{\small Department of Mathematics, University of Padova, 35121 Padova, Italy}
\affil[2]{\small Department of Mathematics and Informatics, University of Calabria, 87036 Rende (Cs), Italy}
\affil[3]{\small Department of Science and High Technology, University of Insubria, 22100 Como, Italy}

%

\maketitle

\begin{abstract}
The algebraic polynomial interpolation on uniformly distributed nodes is
affected by the Runge phenomenon, also when the function to be interpolated
is analytic. Among all techniques that have been proposed to defeat this
phenomenon, there is the mock-Chebyshev interpolation which is an
interpolation made on a subset of the given nodes whose elements mimic
\textit{as well as possible} the Chebyshev-Lobatto points. In this work we use
the simultaneous approximation theory to combine the previous
technique with a polynomial regression in order to increase the accuracy of
the approximation of a given analytic function. We give indications on
how to select the degree of the simultaneous regression in order to obtain
polynomial approximant good in the uniform norm and provide a sufficient condition
to improve, in that norm, the accuracy of the mock-Chebyshev
interpolation with a simultaneous regression. Numerical results are provided.
\end{abstract}

\noindent\textbf{Keywords: }Runge phenomenon; Chebyshev-Lobatto nodes; mock-Chebyshev interpolation; simultaneous regression


\let\thefootnote\relax\footnote{\emph{Email addresses}: \texttt{demarchi@math.unipd.it} (Stefano De Marchi), \texttt{francesco.dellaccio@unical.it} (Francesco Dell'Accio), \texttt{mariarosa.mazza@uninsubria.it} (Mariarosa Mazza)}
\section{Introduction}
In many scientific disciplines, when we want to study a phenomenon, we can start in observing and recording
what happens at regular instants of time. This provides a sample of information that we can use
to give a more or less accurate approximation of the observed phenomenon.
For this aim mathematical tools are needful. The first step is to imagine regular instants
of time as a set of uniform distributed points and the sample of information as the evaluations of an unknown function. In this case a classical technique, used to associate to the discrete set of experimental data a continuous approximation of the phenomenon, is the algebraic polynomial interpolation. This technique has the well-known drawback that on uniformly distributed nodes might not converge, even if the considered function is regular. A classical example is given by Runge's function%
\begin{equation*}
f(t)=\frac{1}{1+25t^{2}}{, ~t\in \lbrack -1,1]}
\end{equation*}%
on an equally spaced triangular array of nodes%
\begin{equation}
x_{0,0};\quad x_{0,1},x_{1,1};\quad x_{0,2},x_{1,2},x_{2,2};\quad%
\ldots \quad ;\quad x_{0,n},x_{1,n},\ldots ,x_{n,n};\quad \ldots   \notag
\end{equation}%
where $x_{i,n}=-1+\frac{2}{n}i$ for $i=0,1,\ldots ,n, n\in{\mathbb N_0}$. In this case, the
error made by interpolating $f$ with polynomials has wild oscillations, a phenomenon known as
\textit{Runge Phenomenon}. Many techniques have been proposed to defeat this
phenomenon; just to mention some of them, the least-squares fitting by
polynomials \cite{rakhmanov2007bounds}, the barycentric rational interpolation \cite{baltensperger1999exponential,bos2011lebesgue,floater2007barycentric}, its extended version \cite{klein2012extension}, the interpolation
on subintervals \cite{boyd2011exponentially}. A further technique exploited to cut down the
Runge phenomenon is the so called mock-Chebyshev subset interpolation, which
takes advantages of the optimality of the interpolation processes on
Chebyshev-Lobatto nodes \cite{rivlin1974chebyshev}. The main goal of this paper consists in a
combination of this kind of interpolation with a regression aimed to improve
the accuracy of the approximation of an analytic function; we will refer to
this combination as \textit{constrained mock-Chebyshev least-squares}.

The paper is structured as follows. In Section 2 we discuss some
details on the mock-Chebyshev subset interpolation. The constrained
mock-Chebyshev least-squares are introduced in the Section 3 and deeply
investigated in Sections 4 and 5 in which we deal with the choice of the
degree of the simultaneous regression and with an estimation of the
error in the uniform norm, respectively. Section 6 is devoted to some
numerical results. Last Section contains the algorithm.
\section{Mock-Chebyshev subset interpolation}

Let $f$ be an analytic function with singularities close to the interval $%
[-1,1]$ and suppose that its evaluations are known on $n+1$ equally spaced points of
that interval. The idea that underlies the
mock-Chebyshev subset interpolation is to interpolate $f$ only on a proper
subset, consisting of $m+1$ of the given nodes, which "looks like" the Chebyshev-Lobatto
grid of order $m+1$. The result is that if we carefully choose $m$, the
convergence of the interpolation process on such a subset of nodes, for $n$
which tends to infinity, will be preserved (cf. \cite{piazzon2013small}). Some notations: from here onwards
we will indicate the equispaced grid of cardinality $n+1$ with the symbol $%
X_n$, while the mock-Chebyshev subset of $X_n$ of order $m+1$ will
be denoted by ${X'_m}$. To understand how to properly choose $m$
(see e.g. \cite{boyd2009divergence}), let us remember that the $m+1$ Chebyshev-Lobatto nodes are
defined as%
\begin{equation*}
x_{j}^{CL}=-\cos \left( \frac{\pi }{m}j\right) ,~j=0,1,\ldots ,m.
\end{equation*}%
Let us expand $x_{1}^{CL}$ in Taylor series centered in zero%
\begin{equation}
x_{1}^{CL}=-1+\frac{\pi ^{2}}{2m^{2}}+O\left( \frac{1}{m^{4}}\right) <-1+%
\frac{\pi ^{2}}{2m^{2}}.  \label{no}
\end{equation}%
Being $x_{0}^{CL}=-1$, the difference $x_{1}^{CL}-x_{0}^{CL}$ is a $O\left(
\frac{1}{m^{2}}\right) $. In other words, this means that the $m+1$ nodes of
Chebyshev-Lobatto are distributed in $[-1,1]$ with a density that is roughly
quadratic in $m$. So for $n$ proportional to $m^{2}$ or $m$
proportional to $\sqrt{n}$, we can select among the given nodes a subset
which mimic a sufficiently large Chebyshev-Lobatto grid. Let $c$ be the
constant of proportionality; a way to calculate it is to impose that the
second node of the Chebyshev-Lobatto grid is as close as possible to the
second node of the equispaced set $X_n$
\begin{equation*}
-\cos \left( \frac{\pi }{m}\right) \simeq -1+\frac{2}{n}.
\end{equation*}%
This can be done in the following manner: by (\ref{no}) we fix the largest
integer $m$ such that
\begin{equation*}
-1+\frac{1}{n}<-1+\frac{\pi ^{2}}{2m^{2}}
\end{equation*}%
so for%
\begin{equation}
m=\left\lfloor \frac{\pi }{\sqrt{2}}\sqrt{n}\right\rfloor  \label{m}
\end{equation}%
for sure $-1+\frac{2}{n}$ is the point of $X_n$ closest to $x_{1}^{CL}$ (for an example, see Figure \ref{mock}).
This choice of $c<\frac{\pi }{\sqrt{2}}$ avoids the fact that the endpoints $%
-1$ and $1$ can be selected more than once.

\begin{figure}
\setlength{\unitlength}{9cm}
\centering
\includegraphics[width=\unitlength]{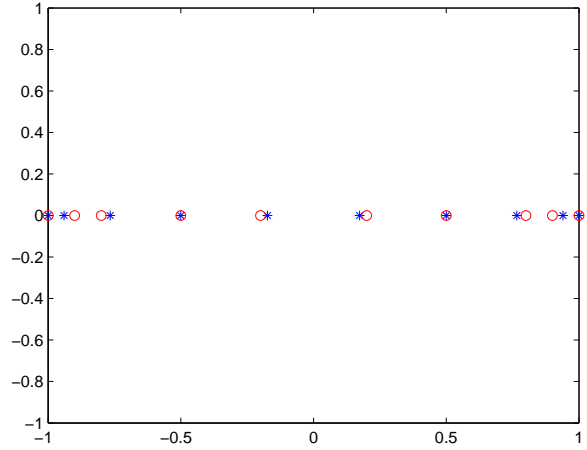}
\caption{Plot of the Chebyshev-Lobatto nodes (\textcolor[rgb]{0.00,0.00,0.98}{$*$}) and mock-Chebyshev nodes (\textcolor[rgb]{0.98,0.00,0.00}{$\circ$}) for $n+1=21, m=\frac{\pi}{\sqrt{2}}\sqrt{20}=9$.} \label{mock}
\end{figure}

For analytic functions the polynomial interpolation on Chebyshev nodes converges geometrically and stably. The mock-Chebyshev interpolation is a stable procedure, but its rate of convergence is subgeometric. In \cite{platte2011impossibility} it has been shown that on equispaced nodes no stable method can converge geometrically.

\section{Constrained mock-Chebyshev least-squares}

In performing the mock-Chebyshev interpolation we know the evaluations
of $f$ on the whole set $X_n$, but actually we only use the information
corresponding to the elements of $X'_m$. Indeed, in \cite{boyd2009divergence}
the $n-m$ remaining nodes are definitively discarded and the corresponding
evaluations are lost. Our idea is to use those nodes, whose set will be
denoted by $X''_{n-m}=\left\{ x''_{1,n-m},x''_{2,n-m},...,x''_{n-m,n-m}\right\}$, $x''_{1,n-m}<x''_{2,n-m}<...<x''_{n-m,n-m}$, to improve the accuracy of the
approximation through a simultaneous regression. More precisely, let $f$ be an
analytic function on $[-1,1]$ and let $\mathcal{P}^{r\ast }=\left\{ P\in \mathcal{P}^{r}:P(x'_{i,m})=f(x'_{i,m}),~i=0,1,\ldots ,m\right\} $ where $\mathcal{P}^{r}$ is the space of polynomials of degree $\le r$ and $m<r\leq n$. We search for the solution of the following \textit{constrained least-squares problem} \cite{gautschi2004orthogonal,mastroianni2008interpolation,bokhari1996sub}%
\begin{equation}
\min_{P\in \mathcal{P}^{r\ast }}\left\Vert f-P\right\Vert _{2}^{2}
\label{probl}
\end{equation}%
where $\left\Vert \cdot \right\Vert _{2}$ is the discrete $2$-norm on $%
X''_{n-m}$.
\begin{theorem}
The constrained least-squares problem (\ref{probl}) has a unique solution.
\end{theorem}
\begin{proof}
Let us denote by $P_{X^{\prime }}$ the interpolating
polynomial for $f$ on $X'_m$. It is not difficult to verify that a generic polynomial $P\in \mathcal{P%
}^{r\ast }$ is of the form $P(t)=P_{X^{\prime }}(t)+Q(t)\omega _{m}(t)$
with $\omega _{m}(t)=\prod\limits_{i=0}^{m}(t-x'_{i,m})$ and $Q(t)$
an arbitrary polynomial of degree $r-m-1$. The problem (\ref{probl}) then becomes
\begin{equation*}
\begin{array}{l}
\displaystyle \min_{Q\in \mathcal{P}^{r-m-1}}\left\Vert f-(P_{X^{\prime }}+Q\omega
_{m})\right\Vert _{2}^{2} \\
\hspace{1cm}=\displaystyle \min_{Q\in \mathcal{P}^{r-m-1}}\sum\limits_{k=1}^{n-m}\left\{
f\left( x''_{k,n-m}\right) -P_{X^{\prime }}\left(
x''_{k,n-m}\right) -Q\left( x''_{k,n-m}\right)
\omega _{m}\left( x''_{k,n-m}\right) \right\} ^{2} \\
\hspace{1cm}=\displaystyle \min_{Q\in \mathcal{P}^{r-m-1}}\sum\limits_{k=1}^{n-m}\left\{
\dfrac{f\left( x''_{k,n-m}\right) -P_{X^{\prime }}\left(
x''_{k,n-m}\right) }{\omega _{m}\left( x''_{k,n-m}\right) }-Q\left( x''_{k,n-m}\right) \right\} ^{2}\omega
_{m}^{2}\left( x''_{k,n-m}\right) .%
\end{array}%
\end{equation*}%
By introducing the following discrete weighted $2$-norm%
\begin{equation*}
\left\Vert u\right\Vert _{2,\omega _{m}^{2}}=\left(
\sum\limits_{k=1}^{n-m}w_{k}u^{2}(x''_{k,n-m})\right) ^{\frac{1%
}{2}}
\end{equation*}%
where $w_{k}=\omega _{m}^{2}(x''_{k,n-m})$ for $k=1,\ldots ,n-m
$ and by defining $\hat{f}$ as%
\begin{equation}
\hat{f}(t):=\frac{f(t)-P_{X^{\prime }}(t)}{\omega _{m}(t)},\quad t\in\left[-1,1\right],\label{fcap}
\end{equation}%
the problem (\ref{probl}) can be reduced to the following classical
least-squares problem%
\begin{equation}
\min\limits_{Q\in \mathcal{P}^{r-m-1}}\left\Vert \hat{f}-Q\right\Vert
_{2,\omega _{m}^{2}}^{2}  \label{probl2}
\end{equation}%
which has a unique solution.
\end{proof}

\smallskip
Denoting by $\hat{Q}_{X^{\prime \prime }}(t)$ the solution of (\ref%
{probl2}), the desired polynomial approximant is
\begin{equation}
\hat{P}_{X}(t)=P_{X^{\prime }}(t)+\hat{Q}_{X^{\prime \prime }}(t)\omega
_{m}(t).\label{pcap}
\end{equation}%
To write $\hat{P}_{X}$ explicitly, let us introduce the discrete inner product
associated to the norm $\left\Vert \cdot \right\Vert _{2,\omega _{m}^{2}}$%
\begin{equation*}
\left( u,v\right) _{\omega
_{m}^{2}}=\sum\limits_{k=1}^{n-m}w_{k}u(x''_{k,n-m})v(x''_{k,n-m})
\end{equation*}%
and let $\left\{ \pi _{i}(t,\omega _{m}^{2})\right\} _{i=0}^{r-m-1}$ be a
basis of $\mathcal{P}^{r-m-1}$ orthogonal with respect to the previous
product. We can express $\hat{Q}_{X^{\prime \prime }}(t)$ with respect to
that basis as%
\begin{equation*}
\hat{Q}_{X^{\prime \prime }}(t)=\sum\limits_{i=0}^{r-m-1}q_{i}\pi
_{i}(t),~q_{i}=\frac{\left( \hat{f},\pi _{i}\right) _{\omega
_{m}^{2}}}{\left( \pi _{i},\pi _{i}\right) _{\omega _{m}^{2}}}.
\end{equation*}%
Then $\hat{P}_{X}(t)$ becomes explicitly%
\begin{equation*}
\hat{P}_{X}(t)=P_{X^{\prime }}(t)+\left( \sum\limits_{i=0}^{r-m-1}q_{i}\pi
_{i}(t)\right) \prod\limits_{i=0}^{m}(t-x'_{i,n}).
\end{equation*}
\begin{theorem}
In the discrete $2$-norm on $X''_{n-m}$ the inequality
\begin{equation*}
\left\Vert f-\hat{P}_{X}\right\Vert _{2}<\left\Vert f-P_{X^{\prime
}}\right\Vert _{2}
\end{equation*}
holds.
\end{theorem}
\begin{proof}
The choice of an orthogonal basis for $\mathcal{P}^{r-m-1}$ allows us to
express the error $\hat{f}-\hat{Q}%
_{X^{\prime \prime }}$ in the $\left\Vert \cdot \right\Vert _{2,\omega
_{m}^{2}}$ norm as follows:%
\begin{equation*}
\left\Vert \hat{f}-\hat{Q}_{X^{\prime \prime }}\right\Vert _{2,\omega
_{m}^{2}}=\left\{ \left\Vert \hat{f}\right\Vert _{2,\omega
_{m}^{2}}^{2}-\sum\limits_{i=0}^{r-m-1} q_{i}^{2}\left\Vert \pi _{i}\right\Vert _{2,\omega _{m}^{2}}^{2}\right\} ^{\frac{1%
}{2}},~q_{i}=\frac{\left( \hat{f},\pi _{i}\right) _{\omega _{m}^{2}}}{\left(
\pi _{i},\pi _{i}\right) _{\omega _{m}^{2}}}.
\end{equation*}%
Therefore the error $f-\hat{P}_{X}$ in the $2$-norm is
\begin{equation*}
\left\Vert f-\hat{P}_{X}\right\Vert _{2}=\left\{ \left\Vert f-P_{X^{\prime
}}\right\Vert _{2}^{2}-\sum\limits_{i=0}^{r-m-1} \tilde{q}%
_{i}^{2}\left\Vert \pi _{i}\omega _{m}\right\Vert
_{2}^{2}\right\} ^{\frac{1}{2}},~\tilde{q}_{i}=\frac{\left(
f-P_{X^{\prime }},\pi _{i}\omega _{m}\right) }{\left( \pi _{i}\omega
_{m},\pi _{i}\omega _{m}\right) }.
\end{equation*}%
\end{proof}

In other words, the error made by using the constrained mock-Chebyshev least-squares method
is, in the $2$-norm, strictly smaller than the error produced when we restrict
ourselves to the mock-Chebyshev subset interpolation.

\section{The degree of simultaneous regression}

As shown in the previous section we approximate the function $f$ with a
least-squares polynomial that satisfies interpolation conditions on a
mock-Chebyshev subset of the given nodes. We have not specified yet how to
choose the degree of the constructed approximant $\hat{P}_{X}$. When this degree increases
up to the total number of nodes the approximation gets worse, since the
combined approximant approaches the interpolating polynomial.

\begin{theorem}
Let $r$ be the degree of $\hat{P}_{X}$ and let us denote by $P_{X}$ the
interpolating polynomial of $f$ on $X_n$. If $r=n$ then%
\begin{equation*}
\hat{P}_{X}\equiv P_{X}.
\end{equation*}
\end{theorem}

\begin{proof}
Recalling that
\begin{equation*}
\hat{P}_{X}(t)=P_{X^{\prime }}(t)+\hat{Q}_{X^{\prime \prime }}(t)\omega
_{m}(t),
\end{equation*}%
if $\hat{P}_{X}$ is an $n$ degree polynomial, the regression polynomial $%
\hat{Q}_{X^{\prime \prime }}$ must be a $n-m-1$ degree polynomial. Since the
least-squares set $X''_{n-m}$ has cardinality $n-m$, $\hat{Q}_{X^{\prime \prime }}$ is the interpolating polynomial for $\hat{f}$ on $X''_{n-m}$ that is%
\begin{equation*}
\hat{Q}_{X^{\prime \prime }}(x''_{k,n-m})=\hat{f}%
(x''_{k,n-m}),~k=1,\ldots ,n-m.
\end{equation*}%
From the previous relation, it follows that%
\begin{eqnarray*}
\hat{P}_{X}(x''_{k,n-m}) &=&P_{X^{\prime }}(x''_{k,n-m})+\hat{Q}_{X^{\prime \prime }}(x''_{k,n-m})\omega
_{m}(x''_{k,n-m}) \\
&=&P_{X^{\prime }}(x''_{k,n-m})+\hat{f}(x''_{k,n-m})\omega _{m}(x''_{k,n-m}) \\
&=&P_{X^{\prime }}(x''_{k,n-m})+\frac{f(x''_{k,n-m})-P_{X^{\prime }}(x''_{k,n-m})}{\omega _{m}(x''_{k,n-m})}\omega _{m}(x''_{k,n-m}) \\
&=&f(x''_{k,n-m})
\end{eqnarray*}%
that is $\hat{P}_{X}$ interpolates $f$ on $X''_{n-m}$.
However, by construction $\hat{P}_{X}$ interpolates also $f$ on $%
X'_m$, then it coincides with the interpolating polynomial for $f$
on $X_n$ by the uniqueness of the interpolating polynomial of degree $n$ on $X_n$.
\end{proof}
\smallskip

By taking into account this result, let us come back to the choice of a proper
degree for $\hat{P}_{X}$. Clearly, it depends on the
degree of the simultaneous regression polynomial, namely of the polynomial $%
\hat{Q}_{X^{\prime \prime }}$. In order to determine a degree for\ $\hat{Q}%
_{X^{\prime \prime }}$ which gives, in the uniform norm, better accuracy of the constrained
mock-Chebyshev least-squares with respect to the mock-Chebyshev
interpolation we use a result presented by L.
Reichel in \cite{reichel1986polynomial}. This result implies that for an equispaced set of $q$\ (internal) nodes of $[-1,1]$%
\begin{equation}
z_{k}=-1+\frac{2k-1}{q},~k=1,\ldots ,q,  \label{nodireich}
\end{equation}%
the degree $p$ of the least-squares polynomial should be selected so that
there is a subset of cardinality $p+1$ of the equispaced set which is close,
in the mock-Chebyshev sense, to the $p+1$ Chebyshev grid. Actually, the
result presented in \cite{reichel1986polynomial} is more general since it deals with the
least-squares approximation of a function on a Jordan curve in the complex
plane. To explain the outlines of Reichel's idea we use his notation.
Let $\Gamma $ be a Jordan curve or Jordan arc in the complex plane and let $\Omega $ the open
set bounded by $\Gamma $. If $\Gamma $ is a Jordan arc then $\Omega$ is void. Let $%
\left\{ z_{k,q}\right\} _{k=1}^{q}$ be a set of $q$ distinct nodes on $\Gamma $.
For a given function $\varphi$ on $\Gamma ,$ let $L_{p,q}\varphi$ denote the
least-squares polynomial of degree $\leq p$ with respect to the semi-norm%
\begin{equation*}
\left\Vert \varphi\right\Vert :=(\varphi,\varphi)^{\frac{1}{2}}
\end{equation*}%
defined through the inner product%
\begin{equation*}
(\varphi,\psi):=\sum\limits_{k=1}^{q}\varphi(z_{k,q})\overline{\psi(z_{k,q})}.
\end{equation*}%
Moreover, let $I_{p}\varphi$ be the interpolating polynomial of $\varphi$ at $p+1$
distinct points $\left\{ w_{k,p}\right\} _{k=0}^{p}$ on $\Gamma .$ We write $I_{p}\prec L_{p,q}$ if $\left\{ w_{k,p}\right\} _{k=0}^{p}\subset \left\{
z_{k,q}\right\} _{k=1}^{q}$. We equip
the domain and the range of $L_{p,q}$ and $I_{p}$ with the uniform norm on $%
\Gamma $
\begin{equation*}
\left\Vert \varphi\right\Vert _{\Gamma }=\sup_{z\in \Gamma }\left\vert
\varphi(z)\right\vert
\end{equation*}%
and we denote the induced operator norm with the symbol $\left\Vert \cdot
\right\Vert $. Finally, we define
\begin{equation*}
E_{p}(\varphi):=\inf_{Q_{p}\in \mathcal{P}^{p}}\left\Vert \varphi-Q_{p}\right\Vert _{\Gamma }.
\end{equation*}%
The following theorem \cite[Theorem 2.1]{reichel1986polynomial} bounds the norm of the least-squares projection $L_{p,q}$ in terms of the norm of the interpolation projection $I_{p}$.

\begin{theorem}
\label{teoreich}
Let $L_{p,q}$ and $I_{p}$ be defined on the set of continuous function on $%
\Gamma \cup \Omega $ and analytic in $\Omega$. Then
\begin{equation}
\left\Vert L_{p,q}\right\Vert \leq \left\Vert I_{p}\right\Vert \left( 1+%
\sqrt{q}\sup_{\left\Vert \varphi\right\Vert _{\Gamma }=1}E_{p}(\varphi)\right) ,~\forall
I_{p}\prec L_{p,q},~\forall q\geq p.  \label{teo}
\end{equation}
\end{theorem}

By means of examples, it has been shown that also when $p$ is fixed the $\sqrt{q}$
growth of the right-hand side of (\ref{teo}) can be achieved. This suggests
to make further assumptions on the distribution of the interpolation nodes
and on the smoothness of the function. Generally, we will assume that $p$ is
an increasing function of $q$. Using a Jackson's theorem \cite[p. 147]{cheney1966introduction} the following
corollary \cite[Corollary 2.1]{reichel1986polynomial} shows that additional smoothness of the function to be
approximated decreases the growth of $\left\Vert L_{p,q}\right\Vert $with $%
q,p(q)$.

\begin{corollary}
\label{correich}
Let $\Gamma =[-1,1]$ and let $F_{d,k,\Gamma}:=\left\{
\varphi:\varphi\in C^{k}[-1,1],~\left\Vert \frac{d^{k}\varphi}{dz^{k}}\right\Vert _{\Gamma
}\leq d\right\} $ be the domain of $L_{p,q}$. Then for some constant $D$ depending on the constant $d$
and on the integer $k$%
\begin{equation*}
\left\Vert L_{p,q}\right\Vert \leq \left\Vert I_{p}\right\Vert \left( 1+D%
\sqrt{q}(p+1)^{-k}\right) ,~\forall I_{p}\prec L_{p,q}.
\end{equation*}
\end{corollary}

The next step is to determine a bound for $\min_{I_{p}\prec
L_{p,q}}\left\Vert I_{p}\right\Vert $. We do not discuss in detail the estimates calculated for $\left\Vert
I_{p}\right\Vert $ in \cite{reichel1986polynomial} but only mention that a useful bound
for $\min_{I_{p}\prec L_{p,q}}\left\Vert I_{p}\right\Vert $ is obtained when
the interpolation points are Fej\'{e}r points or points close to Fej\'{e}r
points. Let us recall that for a generic curve $\Gamma$ the Fej\'{e}r points
are defined as the image on $\Gamma$ of equispaced nodes onto the unit circle through a particular conformal mapping \cite{reichel1986polynomial}. In particular, if $\Gamma =[-1,1]$ the Chebyshev points are Fej\'{e}r points \cite[Example 3.1]{reichel1986polynomial}. The estimates obtained for $\left\Vert
I_{p}\right\Vert $ in \cite{reichel1986polynomial} suggest the following least-squares approximation method:
\begin{criterion}
Let $\Gamma =[-1,1]$. Given a function $\varphi\in F_{d,k,\Gamma}$ and $q$ least-squares nodes $\left\{ z_{k,q}\right\} _{k=1}^{q}$ on $\Gamma $, choose the degree of the approximating polynomial $L_{p,q}\varphi$
as the greatest $p$ such that $p+1$ points are close to $p+1$ Fej\'{e}r points.
\end{criterion}
When the $q$ nodes are equispaced like in (\ref{nodireich}) this means that the degree $p$ of the
least-squares approximant should be selected so that there are $p+1$ points among the
equispaced ones which are close to the $p+1$ Chebyshev nodes. In other
words, $p$ should be selected in the mock-Chebyshev sense.

In the case of simultaneous regression the least-squares nodes are those of $X''_{n-m}$ and therefore they are not equally spaced. However, when the cardinality of $X_n$ is sufficiently large we can approximate an equispaced grid with width $\geq 2h$, $h=\frac{2}{n}$ using nodes belonging to $X''_{n-m}$. In fact, the maximum distance between two consecutive
nodes of $X''_{n-m}$ is at most $2h$. To be aware of it, let us observe that the interval $I=\left[ x''_{1,n-m},x''_{n-m,n-m}\right]$ according to the mock-Chebyshev extraction is properly contained in $[-1,1]$ and symmetric with respect to the origin. Because of the choice of $m$ the first and the second node of
$X'_m$ are equal to $x_{0,n}$ and $x_{1,n}$, respectively, i.e. $X'_m=\left\{x_{0,n},x_{1,n},\dots\right\}$.
Moreover, we have
\begin{lemma}\label{lem1}
The first three nodes of $X_n$ belong to $%
X'_m$, i.e.
\begin{equation*}
X'_m=\left\{x_{0,n},x_{1,n},x_{2,n},\dots\right\}.\label{xsec}
\end{equation*}
\end{lemma}
\begin{proof}
To prove that $x_{2,n}$ together with $x_{0,n},x_{1,n}$ has been taken during the
mock-Chebyshev extraction, we need to expand in Taylor series the
difference between the second and the third Chebyshev-Lobatto node%
\begin{eqnarray*}
x_{2}^{CL}-x_{1}^{CL} &=&-\cos \left( \frac{2\pi }{m}\right) +\cos \left(
\frac{\pi }{m}\right)=-2\sin \left( \frac{3\pi }{2m}\right) \sin \left( -\frac{\pi }{2m}\right)=2\frac{\pi }{2m}\frac{3\pi }{2m}+O\left( \frac{\pi ^{4}}{m^{4}}\right)<2\frac{\pi }{2m}\frac{3\pi }{2m}.
\end{eqnarray*}%
Recalling that $m$ is given by (\ref{m}) the previous difference can be
rounded up by $\frac{3}{n}$ and the thesis follows (see Figure \ref{fig1}).
\end{proof}
\begin{figure}
\setlength{\unitlength}{11cm}
\centering
\includegraphics[width=\unitlength]{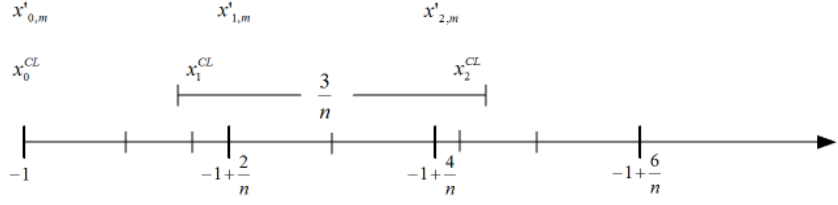}
\caption{Proof of Lemma \ref{lem1}}\label{fig1}
\end{figure}
\begin{lemma}\label{lem2}
For $n>7$, $x_{3,n}$ does not belong to $X'_m$, i.e.
\begin{equation*}
x_{3,n}\in X''_{n-m}, \quad n>7.
\end{equation*}
\end{lemma}
\begin{proof}
Let us expand $x_{3}^{CL}$\ in Taylor series%
\begin{eqnarray*}
x_{3}^{CL} &=&-\cos \left( \frac{3\pi }{m}\right) =-1+\frac{9\pi ^{2}%
}{2m^{2}}-\frac{81\pi ^{4}}{24m^{4}}+O\left( \left( \frac{3\pi }{m}\right)
^{6}\right)>-1+\frac{9}{n}-\frac{27}{2n^{2}}
\end{eqnarray*}%
and check for which values of $n\in
\mathbb{N}
$ the following inequality holds%
\begin{equation*}
-1+\frac{9}{n}-\frac{27}{2n^{2}}>-1+\frac{7}{n}.
\end{equation*}%
We obtain that
\begin{equation*}
n>\frac{27}{4}
\end{equation*}%
and therefore $\left|x_{3}^{CL}+1-\frac{6}{n}\right|>\left|x_{3}^{CL}+1-\frac{8}{n}\right|$.
\end{proof}

\bigskip
\begin{proposition}
For sufficiently large $n$ the following inequality
\begin{equation*}
\max_{2\leq i\leq n-m}\left\vert x''_{i-1,n-m}-x''_{i,n-m}\right\vert \leq 2h
\end{equation*}
holds.
\end{proposition}

\begin{proof}
The thesis is equivalent to the fact that among the nodes of $X'_m$ belonging to $I=\left[-1+\frac{6}{n},1-\frac{6}{n}\right]$ there are not two consecutive nodes of $X_n$.
By Lemma \ref{lem1} and Lemma \ref{lem2} the nodes of the $m+1$ Chebyshev-Lobatto grid which are contained in $I$
are%
\begin{equation}
x_{j}^{CL}=-\cos \left( \frac{\pi }{m}j\right) ,~j=3,\ldots ,m-3.~
\label{nodes}
\end{equation}%
It is well-known that the nodes (\ref{nodes}) are more dense near the endpoints of $
I $ and less near its center, therefore it is sufficient to verify that the distance between $x_{3}^{CL}$ and $%
x_{4}^{CL}$ is greater than $2h$. Let us expand in Taylor series $%
x_{4}^{CL}-x_{3}^{CL}$
\begin{eqnarray*}
x_{4}^{CL}-x_{3}^{CL} &=&-\cos \left( \frac{4\pi }{m}\right) +\cos \left(
\frac{3\pi }{m}\right)=-2\sin \left( \frac{7\pi }{2m}\right) \sin \left( -\frac{\pi }{2m}\right)
\\
&=&2\left( \frac{7\pi}{2m}-\left( \frac{7\pi }{2m}\right) ^{3}\frac{1}{6}%
+O\left( \left( \frac{7\pi }{2m}\right) ^{5}\right) \right) \left( \frac{\pi
}{2m}-\left( \frac{\pi }{2m}\right) ^{3}\frac{1}{6}+O\left( \left( \frac{%
7\pi }{2m}\right) ^{5}\right) \right)  \\
&=&\frac{7\pi ^{2}}{2m^{2}}-\frac{175\pi ^{4}}{24m^{4}}+O\left( \frac{\pi
^{6}}{m^{6}}\right)
\end{eqnarray*}
round downward by%
\begin{equation*}
\frac{7}{n}-\frac{175}{6n^{2}}<x_{4}^{CL}-x_{3}^{CL}
\end{equation*}%
and impose that
\begin{equation*}
\frac{4}{n}<\frac{7}{n}-\frac{175}{6n^{2}}.
\end{equation*}%
From the previous inequality it follows that%
\begin{equation*}
 n>\frac{175}{18}\simeq 9.72
\end{equation*}%
and the thesis holds.
\end{proof}
\bigskip

At this point we can apply the results presented in \cite{reichel1986polynomial} to the simultaneous regression. Taking into account that the grid (\ref{nodireich}) is equispaced in $\left[-1+\frac{1}{q},1-\frac{1}{q}\right]$ with width $\frac{2}{q}$, we note that, for $n$ sufficiently large, we can approximate such a grid with $q=\frac{n}{6}=\frac{1}{3h}$ and nodes coming from $X''_{n-m}$. We denote this grid with $\tilde{X}''_{n-m}$. The choice for the degree of the simultaneous
regression which gives good approximation in the uniform norm is therefore
\begin{equation}
p=\left\lfloor \frac{\pi }{\sqrt{2}}\sqrt{q}\right\rfloor =\left\lfloor
\frac{\pi }{\sqrt{2}}\sqrt{\frac{n}{6}}\right\rfloor .  \label{p}
\end{equation}%
Let us observe that since the degree of the mock-Chebyshev interpolation and
the degree of the regression are chosen in the same way, we can obtain the
previous result applying to $X''_{n-m}$ the idea explained in
\cite{boyd2009divergence}, that is imposing that
\begin{equation*}
-\cos \left( \frac{\pi }{p}\right) \simeq -1+\frac{6}{n}.
\end{equation*}%
It is a straightforward calculus to prove that $p$ will be like in (\ref{p}).

\section{Uniform norm estimation}

We have determined the degree $p$ as in (\ref{p}) for the polynomial $\hat{Q}_{X^{\prime \prime }}$
which, according to\ Reichel's theory, gives good approximation in the
uniform norm. Now, we want to calculate an estimation for the norm error $E_{\hat{P}_{X}}(f)=\left \Vert f-\hat{P}_{X}\right \Vert_\infty$ in the uniform norm. Let $\hat{P}_{X}:C[-1,1]\rightarrow \mathcal{P}^{r\ast }$ the projection operator which
associates to a continuous function in $[-1,1]$ its constrained
mock-Chebyshev polynomial and let $\hat{Q}_{X^{\prime \prime }}:C[-1,1]\rightarrow \mathcal{P}^{r-m-1}$ the projection
operator which associates to a continuous function in $[-1,1]$ its least-squares polynomial in the norm $\left\Vert \cdot \right\Vert _{2,\omega
_{m}^{2}}$.

As in the proof of Theorem \ref{teoreich} and Corollary \ref{correich} we get an estimate for the operator norm $\left\Vert \hat{Q}_{X^{\prime \prime }}\right\Vert $.

\begin{theorem}
Let $\varphi\in C[-1,1]$ and $I_{p}\varphi$ be the interpolating polynomial
of $\varphi$ on the $p+1$ mock-Chebyshev subset $X'''_p=\left\{ x_{k,p}^{\prime \prime \prime }\right\} _{k=0}^{p}$ of $\tilde{X}''_{n-m}$. Then
\[
\left\Vert \hat{Q}_{X^{\prime \prime }}\right\Vert \leq \left\Vert
I_{p}\right\Vert \left( 1+\frac{\left( \sum\limits_{k=1}^{n-m}w_{k}\right) ^{%
\frac{1}{2}}}{\min\limits_{j=0,...,p}\sqrt{\tilde{w}_{j}}}%
\sup_{\left\Vert \varphi\right\Vert _{\infty }=1}E_{p}(\varphi)\right).
\]
\end{theorem}

\begin{proof}
Let $Q_{p}^{\ast }\varphi$ be the polynomial of degree $\leq p$ such that $%
E_{p}(\varphi)=\left\Vert \varphi-Q_{p}^{\ast }\varphi\right\Vert _{\infty }$%
. By (\ref{probl2})%
\[
\left\Vert \hat{Q}_{X^{\prime \prime }}\varphi-\varphi\right\Vert _{2,\omega
_{m}^{2}}\leq \left\Vert Q_{p}^{\ast }\varphi-\varphi\right\Vert _{2,\omega
_{m}^{2}}.
\]%
On the other hand,

\begin{equation}
\begin{array}{c}
\left\Vert Q_{p}^{\ast }\varphi-\varphi\right\Vert _{2,\omega
_{m}^{2}}=\left( \sum\limits_{k=1}^{n-m}w_{k}\left( Q_{p}^{\ast
}(x''_{k,n-m})-\varphi(x''_{k,n-m})\right)
^{2}\right) ^{\frac{1}{2}} \\
\leq \left( \sum\limits_{k=1}^{n-m}w_{k}\right) ^{\frac{1}{2}}\left\Vert
Q_{p}^{\ast }\varphi-\varphi\right\Vert _{\infty } \\
=\left( \sum\limits_{k=1}^{n-m}w_{k}\right) ^{\frac{1}{2}}E_{p}(\varphi).%
\end{array}
\label{eq3}
\end{equation}%
Let $l_{k}(t)$ $k=0,\dots,p$ be the elementary Lagrangian polynomials associated with $%
X'''_p$, that is
\[
I_{p}\varphi(t)=\sum\limits_{j=0}^{p}\varphi(x'''_{j,p})l_{j}(t).
\]%
Let us express $\hat{Q}_{X^{\prime \prime }}\varphi$ in the same basis as%
\[
\hat{Q}_{X^{\prime \prime }}\varphi(t)=\sum\limits_{j=0}^{p}\alpha
_{j}l_{j}(t),
\]%
for some coefficients $\alpha _{j}$. From (\ref{eq3}) it follows that%
\[
\sqrt{\widetilde{w}_{j}}\left\vert \alpha _{j}-\varphi(x'''_{j,p})\right\vert
\leq \left\Vert \hat{Q}_{X^{\prime \prime }}\varphi-\varphi\right\Vert
_{2,\omega _{m}^{2}}\leq \left( \sum\limits_{k=1}^{n-m}w_{k}\right) ^{\frac{1%
}{2}}E_{p}(\varphi),~j=0,\ldots ,p,
\]%
where $\widetilde{w}_{j},$ $j=0,\ldots ,p$ are the positive weights corresponding to
the nodes $\left\{ x'''_{k,p}\right\} _{k=0}^{p}$ and then%
\[
\left\vert \alpha _{j}-\varphi(x'''_{j,p})\right\vert \leq \frac{\left(
\sum\limits_{k=1}^{n-m}w_{k}\right) ^{\frac{1}{2}}}{\sqrt{\widetilde{w}_{j}}}%
E_{p}(\varphi).
\]%
Substituting the previous relation into%
\[
\left\vert \hat{Q}_{X^{\prime \prime }}\varphi(t)\right\vert \leq
\sum\limits_{j=0}^{p}\left\vert \alpha _{j}-\varphi(x'''_{j,p})\right\vert
\left\vert l_{j}(t)\right\vert +\sum\limits_{j=0}^{p}\left\vert
\varphi(x'''_{j,p})\right\vert \left\vert l_{j}(t)\right\vert,
\]%
we obtain%
\[
\left\Vert \hat{Q}_{X^{\prime \prime }}\right\Vert =\sup_{\left\Vert
\varphi\right\Vert _{\infty }=1}\left\Vert \hat{Q}_{X^{\prime \prime
}}\varphi\right\Vert _{\infty }\leq \left\Vert I_{p}\right\Vert \frac{\left(
\sum\limits_{k=1}^{n-m}w_{k}\right) ^{\frac{1}{2}}}{\min\limits_{j=0,...,p}%
\sqrt{\widetilde{w}_{j}}}\sup_{\left\Vert \varphi\right\Vert _{\infty
}=1}E_{p}(\varphi)+\left\Vert I_{p}\right\Vert
\]%
which proves the theorem.
\end{proof}

Recall that, fixed $\Gamma=[-1,1]$, according to \cite{reichel1986polynomial}, for each $k\in {\mathbb N}$ and $d>0$ we set
\begin{equation*}
F_{d,k,\Gamma}:=\left\{
\varphi :\varphi \in C^{k}[-1,1],~\left\Vert \frac{d^{k}\varphi }{dz^{k}}%
\right\Vert _{\Gamma }\leq d\right\}.
\end{equation*}
\begin{corollary}
\label{cornos}If $\hat{Q}_{X^{\prime \prime }}$ has domain $F_{d,k,\Gamma}$ there exists a constant $D$ depending
on $d$ and on the integer $k$ such that%
\begin{equation}
\left\Vert \hat{Q}_{X^{\prime \prime }}\right\Vert \leq \left\Vert
I_{p}\right\Vert \left( 1+D\frac{\left( \sum\limits_{k=1}^{n-m}w_{k}\right)
^{\frac{1}{2}}}{\min\limits_{j=0,...,p}\sqrt{\widetilde{w}_{j}}}(p+1)^{-k}%
\right).  \label{cor}
\end{equation}
\end{corollary}

\begin{proof}
From a Jackson's theorem \cite[p. 147]{cheney1966introduction} for $\varphi\in F_{d,k,\Gamma}$ it follows
\[
E_{p}(\varphi)\leq D(p+1)^{-k}
\]%
where $D$ is a constant depending on $d$ and on the integer $k$.
\end{proof}

With these results in mind we can provide an estimate in the uniform norm for the
error of the constrained mock-Chebyshev least-squares.

\begin{theorem}
\label{teonos}Let $f\in F_{d,p,\Gamma}$. Then
\begin{equation}\label{est}
E_{\hat{P}_{X}}(f)\leq \left( 1+\left\Vert
I_{p}\right\Vert \left( 1+D\frac{\left( \sum\limits_{k=1}^{n-m}w_{k}\right)
^{\frac{1}{2}}}{\min\limits_{j=0,...,p}\sqrt{\widetilde{w}_{j}}}(p+1)^{-p}%
\right) \right) E_{p}(\hat{f})\left\Vert \omega _{m}\right\Vert _{\infty }.
\end{equation}
\end{theorem}

\begin{proof}
Let us start from the following relations
\begin{eqnarray*}
E_{\hat{P}_{X}}(f) &=&\left \Vert f-P_{X^{\prime}}f-\hat{Q}_{X^{\prime \prime }}\hat{f}\omega _{m} \right\Vert_\infty\\
&=&\left \Vert \frac{f-P_{X^{\prime}}f}{\omega _{m}}\omega _{m}-\hat{Q}_{X^{\prime \prime }}\left( \frac{f-P_{X^{\prime}}f}{\omega _{m}}\right)\omega _{m}\right \Vert_\infty \\
&\leq&E_{\hat{Q}_{X^{\prime \prime }}}\left( \frac{f-P_{X^{\prime}}f}{\omega _{m}%
}\right) \left\Vert\omega _{m}\right\Vert_\infty
\end{eqnarray*}%
where $E_{\hat{Q}_{X^{\prime \prime }}}\left( \frac{f-P_{X^{\prime}}f}{\omega
_{m}}\right)$ is the uniform norm error made in approximating $\hat{f}$ with its least-squares polynomial in the norm $\left\Vert \cdot \right\Vert _{2,\omega
_{m}^{2}}$. Since $\hat{Q}_{X^{\prime \prime }}$ is a projection operator
which reproduces the polynomials the following inequality holds%
\[
E_{\hat{Q}_{X^{\prime \prime }}}\left( \frac{f-P_{X^{\prime}}f}{%
\omega _{m}}\right)\leq \left( 1+\left\Vert \hat{Q}%
_{X^{\prime \prime }}\right\Vert \right) E_{p}(\hat{f})
\]%
where $E_{p}(\hat{f})=\min\limits_{Q\in \mathcal{P}^{p}}\left\Vert \hat{f}%
-Q\right\Vert _{\infty }$. Therefore%
\begin{eqnarray*}
E_{\hat{P}_{X}}(f)&\leq &\left( 1+\left\Vert \hat{Q}_{X^{\prime \prime }}\right\Vert \right)
E_{p}(\hat{f})\left\Vert \omega _{m}\right\Vert _{\infty }
\end{eqnarray*}%
which applying Corollary \ref{cornos} to $f$ gives the thesis.
\end{proof}

Theorem \ref{teonos} gives a sufficient condition to improve in the uniform norm the
accuracy of the mock-Chebyshev interpolation through the constrained
mock-Chebyshev least-squares.

\begin{corollary}
Let $f\in C^{m+1}[-1,1]$. If
\[
\left( 1+\left\Vert I_{p}\right\Vert \left( 1+D\frac{\left(
\sum\limits_{k=1}^{n-m}w_{k}\right) ^{\frac{1}{2}}}{\min\limits_{j=0,...,p}%
\sqrt{\widetilde{w}_{j}}}(p+1)^{-p}\right) \right) E_{p}(\hat{f})<\frac{%
\left\Vert f^{(m+1)}\right\Vert }{(m+1)!}
\]%
then%
\[
E_{\hat{P}_{X}}(f)<E_{P_{X^{\prime }}}(f)
\]%
where $E_{P_{X^{\prime }}}(f)=\left\Vert f-P_{X^{\prime }}\right\Vert_\infty$.
\end{corollary}

\begin{proof}
Let us recall that the error in the Lagrange interpolation can be bounded
as follows%
\[
E_{P_{X^{\prime }}}(f)\leq \frac{\left\Vert
f^{(m+1)}\right\Vert }{(m+1)!}\left\Vert \omega _{m}\right\Vert _{\infty }.
\]%
From Theorem \ref{teonos} we get the thesis.
\end{proof}

Finally, the following corollary shows that the operator $\hat{P}_{X}$ reproduces polynomials of degree $\leq m+p$.
\begin{corollary}
If $f=p_r$ with $p_r\in\mathcal{P}^{m+p}$, then
\begin{equation*}
\hat{P}_{X}f=f.
\end{equation*}
\end{corollary}
\begin{proof}
If $f=p_{r}$ with $r\leq m$%
\begin{equation*}
\hat{f}(t)=\frac{p_{r}(t)-P_{X^{\prime }}p_{r}(t)}{\omega _{m}(t)}=\frac{%
p_{r}^{(m+1)}(\xi _{t})}{(m+1)!}\equiv 0.
\end{equation*}
If $f=p_{r}$ with $m<r\leq m+p$%
\begin{equation*}
\hat{f}(t)=\frac{p_{r}(t)-P_{X^{\prime }}p_{r}(t)}{\omega _{m}(t)}
\end{equation*}
is a polynomial of degree $r-(m+1)$.
In both cases $E_{p}(\hat{f})=0$ and the right-hand side of (\ref{est}) is zero.
\end{proof}

\section{Numerical results}
We finally carried out a series of numerical tests to compare, in the uniform norm, the approximation of the constrained mock-Chebyshev least-squares and the mock-Chebyshev interpolation.
A first set of test functions includes the following ones (the first three functions were already considered in \cite{berzins2007adaptive}):%
\begin{equation*}
\begin{array}{l}
f_{1}(t)=\sqrt{\left\vert t\right\vert }, \\
\\
f_{2}(t)=\frac{1}{1+25t^{2}}, \\
\\
f_{3}(t)=\frac{10^{-15}}{10^{-15}+25t^{2}}, \\
\\
f_{4}(t)=t\left\vert t\right\vert, \\
\end{array}%
~~{\normalsize t\in \lbrack -1,1]}.
\end{equation*}%
The function $f_{1}$ is H\"{o}lder continuous with exponent $1/2$, the function $f_{3}$ is a modification of $f_{2}$ obtained by introducing the exponential $10^{-15}$ in order to squash $f_{2}$ on $x$ and $y$ axes, the function $f_{4}$ is of class $C^1$. The errors are computed as the maximum absolute value of the difference between the approximant and the exact function at $10001$ equispaced points in $[-1,1]$. Let us rename with $p$
the degree of the simultaneous regression polynomial $\hat{Q}_{X^{\prime
\prime }}$.
\begin{table}[h]
\centering{\scriptsize
\begin{tabular}{ccccc}
\hline
$p$ & $E_{\hat{P}_{X}}(f_{1})$ & $E_{\hat{P}_{X}}(f_{2})$ & $%
E_{\hat{P}_{X}}(f_{3})$ & $E_{\hat{P}_{X}}(f_{4})$ \\ \hline
\multicolumn{1}{l}{} & \multicolumn{1}{l}{} & \multicolumn{1}{l}{} &
\multicolumn{1}{l}{} &  \\
$\vdots $ & $\vdots $ & $\vdots $ & $\vdots $ & $\vdots $ \\
$28$ & \multicolumn{1}{l}{$\textcolor[rgb]{0.00,0.98,0.00}{7.9726586e-002}$}
& \multicolumn{1}{l}{$\textcolor[rgb]{0.00,0.98,0.00}{9.7493857e-009}$} &
\multicolumn{1}{l}{$\textcolor[rgb]{0.00,0.98,0.00}{9.9994994e-001}$} &
\multicolumn{1}{l}{$\textcolor[rgb]{0.00,0.98,0.00}{5.4308526e-005 }$} \\
$29$ & $7.8915085e-002$ & $8.5899644e-009$ & $9.9994769e-001$ & $%
5.4308526e-005$ \\
$\vdots $ & $\vdots $ & $\vdots $ & $\vdots $ & $\vdots $ \\
$33$ & $7.7268588e-002$ & $6.2480174e-009$ & $9.9994276e-001$ & $%
4.8879070e-005$ \\
$34$ & $7.7268642e-002$ & $6.2483426e-009$ & $9.9994277e-001$ &
\multicolumn{1}{l}{$\textcolor[rgb]{0.98,0.00,0.00}{4.6554802e-005}$} \\
$35$ & $7.7593676e-002$ & $7.6886833e-009$ & $9.9994378e-001$ & $%
4.6554852e-005$ \\
$36$ & $7.7593662e-002$ & $7.6886787e-009$ & $9.9994377e-001$ & $%
4.8513243e-005$ \\
$37$ & $7.6667437e-002$ & \multicolumn{1}{l}{$%
\textcolor[rgb]{0.98,0.00,0.00}{5.8468658e-009}$} & $9.9994084e-001$ & $%
4.8512907e-005$ \\
$38$ & $7.6667394e-002$ & $5.8470333e-009$ & $9.9994083e-001$ & $%
5.0626752e-005$ \\
$\vdots $ & $\vdots $ & $\vdots $ & $\vdots $ & $\vdots $ \\
$47$ & $7.5926645e-002$ & $7.2563305e-009$ & $9.9993836e-001$ & $%
7.8662677e-005$ \\
$48$ & \multicolumn{1}{l}{$\textcolor[rgb]{0.98,0.00,0.00}{7.5926555e-002}$}
& \multicolumn{1}{l}{$7.2566879e-009$} & \multicolumn{1}{l}{$9.9993834e-001$}
& $8.3106886e-005$ \\
$49$ & $7.6081471e-002$ & $8.0118418e-009$ & $9.9993892e-001$ & $%
8.3106580e-005$ \\
$\vdots $ & $\vdots $ & $\vdots $ & $\vdots $ & $\vdots $ \\
$59$ & $9.8058844e-002$ & $9.7826094e-009$ & $9.9993832e-001$ & $%
1.2059132e-004$ \\
$60$ & $9.8061139e-002$ & $9.7829010e-009$ & \multicolumn{1}{l}{$%
\textcolor[rgb]{0.98,0.00,0.00}{9.9993831e-001}$} & $1.2342356e-004$ \\
$61$ & $1.0514604e-001$ & $1.1889342e-008$ & $9.9993920e-001$ & $%
1.2342492e-004$ \\
$\vdots $ & $\vdots $ & $\vdots $ & $\vdots $ & $\vdots $ \\
$99$ & $3.5158374e-001$ & $2.9978376e-008$ & $3.0304570e+000$ & $%
3.8993185e-004$ \\
$100$ & $3.5157737e-001$ & $2.9977317e-008$ & $3.0304057e+000$ & $%
4.0022643e-004$ \\
\multicolumn{1}{l}{} & \multicolumn{1}{l}{} & \multicolumn{1}{l}{} &
\multicolumn{1}{l}{} &  \\ \hline
\multicolumn{1}{l}{} & $E_{P_{X^{\prime }}}(f_{1})$ & $E_{P_{X^{\prime }}}(f_{2})$ & $%
E_{P_{X^{\prime }}}(f_{3})$ & $E_{P_{X^{\prime }}}(f_{4})$ \\ \hline
\multicolumn{1}{l}{} & \multicolumn{1}{l}{} & \multicolumn{1}{l}{} &
\multicolumn{1}{l}{} &  \\
\multicolumn{1}{l}{} & \multicolumn{1}{l}{$%
\textcolor[rgb]{0.00,0.00,0.98}{8.7569583e-002}$} & \multicolumn{1}{l}{$%
\textcolor[rgb]{0.00,0.00,0.98}{8.9863528e-007}$} & \multicolumn{1}{l}{$%
\textcolor[rgb]{0.00,0.00,0.98}{9.9996656e-001}$} & \multicolumn{1}{l}{$%
\textcolor[rgb]{0.00,0.00,0.98}{1.5095571e-004}$}%
\end{tabular}}
\caption{Comparison between $E_{\hat{P}_{X}}(f_{i})$ and $E_{P_{X^{\prime }}}(f_{i})$ for $n=1000$. In this case $m=70$, $p^{\ast }=28$.} \label{tabella}
\end{table}
In Table \ref{tabella} $p$ ranges from $p=28$ to $p=100$. We denote with $p^{\ast }$ the degree of the simultaneous regression which, according to the theory explained above, gives good approximation in
the uniform norm. Table \ref{tabella} allows to compare the two errors of interest in the case of $n+1=1001$ equispaced interpolation nodes. At the
top of the table, in green, is highlighted the error $E_{\hat{P}_{X}}(f_{i})$ in correspondence of the degree $p^{\ast }$.
In red is highlighted the minimum possible error $E_{\hat{P}_{X}}(f_{i})$ in the range $[1,n-m-1]$. At the bottom,
in blue, is represented the error $E_{P_{X^{\prime }}}(f_{i})$. As we can see, the constrained mock-Chebyshev least-squares improve the accuracy of the approximation of the mock-Chebyshev interpolation. We note that
in correspondence of the degree $p^{\ast }$ we obtain an improvement of the accuracy of approximation. More in detail, for $f_{1}$ there is an interval for $p$ in which the
approximation obtained with our method is better than the one coming from
the mock-Chebyshev interpolation. In this case the improvement involves only
the coefficients. When the function to be approximated is the Runge function,
our approximation is everywhere more accurate for $p$ ranging from $1$ to $%
100$. In particular, there is a range for $p$ in which we get $2$ digits of precision
more than the mock-Chebyshev interpolation and $p^{\ast }$ lies in this range. For $f_{3}$
our approximation is, up to a certain value, better but almost the same of the
approximation obtained with the mock-Chebyshev interpolation and then gets little worse. In the case of $f_{4}$ there is an interval for $p$ in which we get $1$ digits of precision more than the mock-Chebyshev interpolation.

We have done further tests using the Runge function and the following ones:
\begin{equation*}
\begin{array}{l}
\\
f_{5}(t)=\frac{1}{t^{2}-(1+0.5)},  \\
\\
f_{6}(t)=\frac{1}{t^{4}+\left( \frac{\sqrt{26}}{5}-1\right) t^{2}+\left(
\frac{13}{50}\right) ^{2}},\\
\\
f_{7}(t)=\frac{1}{t^4+\left(\frac{2}{50}\right)^{2}},
\end{array}%
~~t\in \lbrack -1,1],
\end{equation*}
which, as the Runge function, are analytic in the interval $[-1,1]$. The function $f_{5}$ has poles at $\pm\sqrt{1+0.5}$, while the function $f_{6}$ has poles at $\frac{1}{5}\pm i\frac{1}{10}$ and $-\frac{1}{5}\pm i\frac{1}{10}$ and the function $f_{7}$ has poles at $\frac{1}{5\sqrt{2}}\pm i\frac{1}{5\sqrt{2}}$ and $-\frac{1}{5\sqrt{2}}\pm i\frac{1}{5\sqrt{2}}$.
\begin{figure}[htb]
\centering
\begin{minipage}[l]{.40\textwidth}
\includegraphics[width=7cm, height=5cm]{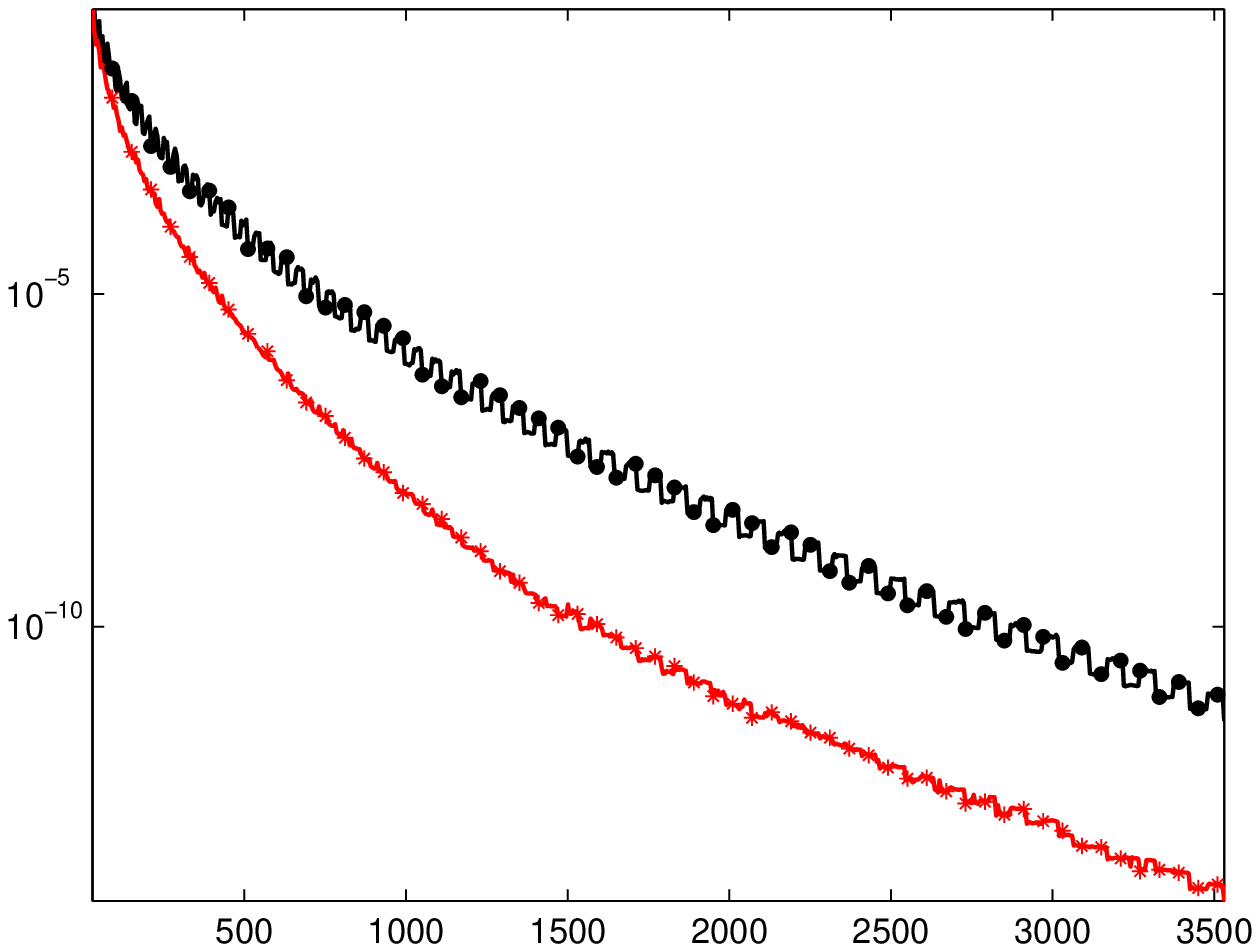}
\caption{Comparison between $E_{\hat{P}_{X}}(f_{2})$ (\textcolor[rgb]{0.98,0.00,0.00}{$*$}) (lower curve) and $E_{P_{X^{\prime }}}(f_{2})$ (\textcolor[rgb]{0.00,0.00,0.00}{$\bullet$}) (upper curve) for $30\leq n\leq 3530$. When $n=3530$, $\deg(\hat{P}_{X}f_{2})=m+p^{*}=131+53$.}\label{runge}
\end{minipage}%
\hspace{10mm}%
\begin{minipage}[l]{.40\textwidth}
\includegraphics[width=7cm, height=5cm]{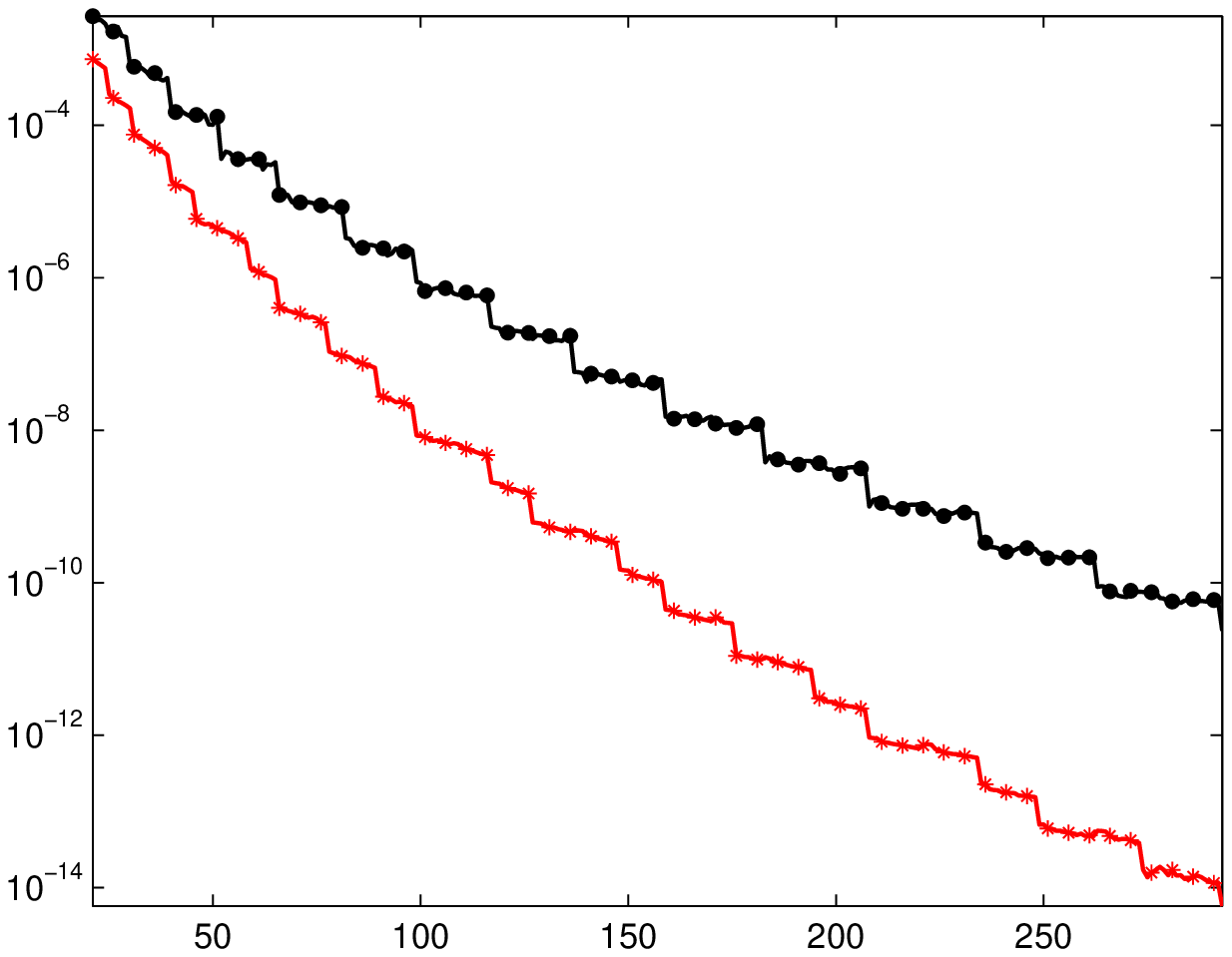}
\caption{Comparison between $E_{\hat{P}_{X}}(f_{5})$ (\textcolor[rgb]{0.98,0.00,0.00}{$*$}) (lower curve) and $E_{P_{X^{\prime }}}(f_{5})$ (\textcolor[rgb]{0.00,0.00,0.00}{$\bullet$}) (upper curve) for $20\leq n\leq 292$. When $n=292$, $\deg(\hat{P}_{X}f_{5})=m+p^{*}=37+15$.}\label{caso32}
\end{minipage}
\\
\begin{minipage}[l]{.40\textwidth}
\includegraphics[width=7cm, height=5cm]{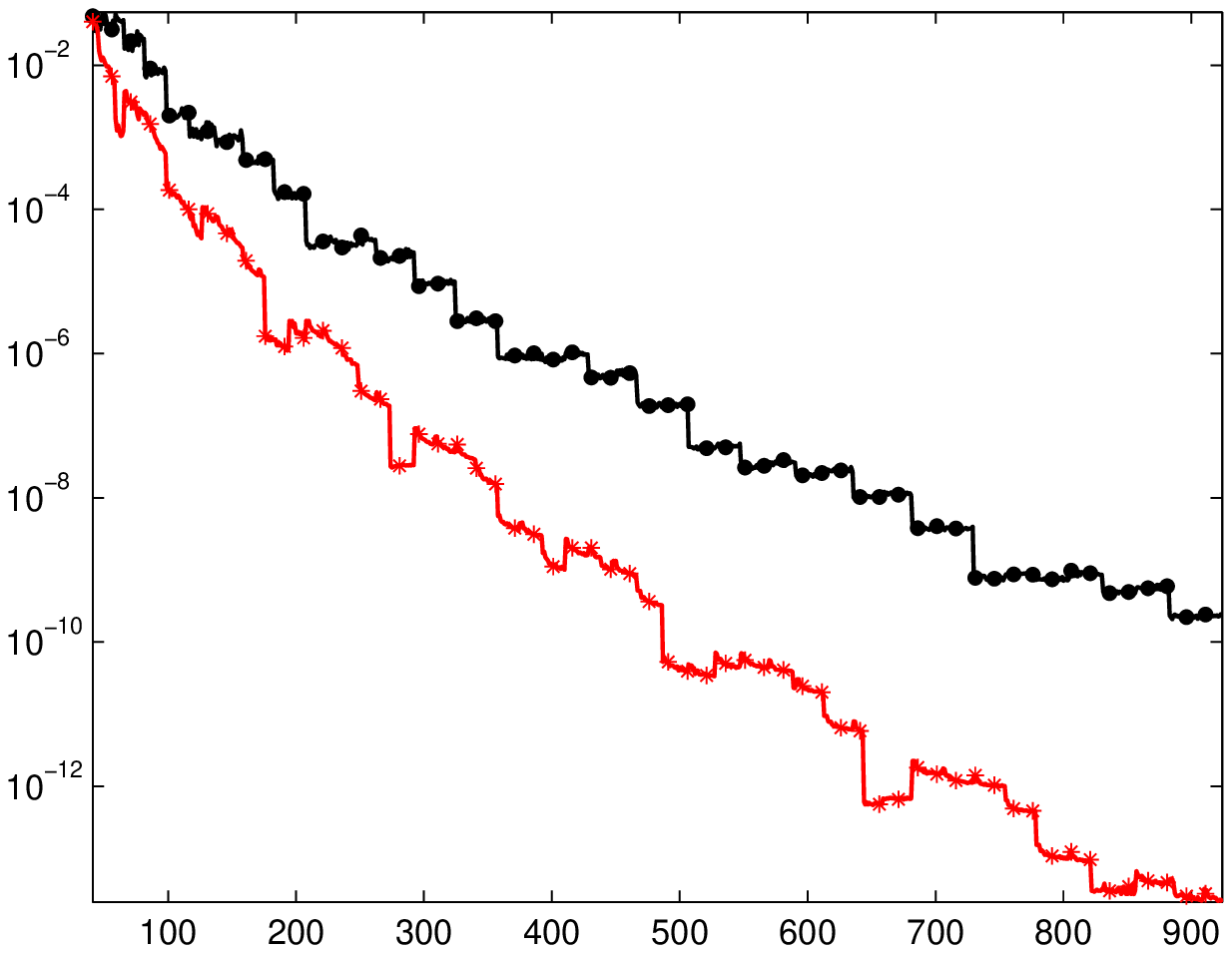}
\caption{Comparison between $E_{\hat{P}_{X}}(f_{6})$ (\textcolor[rgb]{0.98,0.00,0.00}{$*$}) (lower curve) and $E_{P_{X^{\prime }}}(f_{6})$ (\textcolor[rgb]{0.00,0.00,0.00}{$\bullet$}) (upper curve) for $40\leq n\leq 924$. When $n=923$, $\deg(\hat{P}_{X}f_{6})=m+p^{*}=67+27$.}\label{caso38}
\end{minipage}
\hspace{10mm}
\begin{minipage}[l]{.40\textwidth}
\includegraphics[width=7cm, height=5cm]{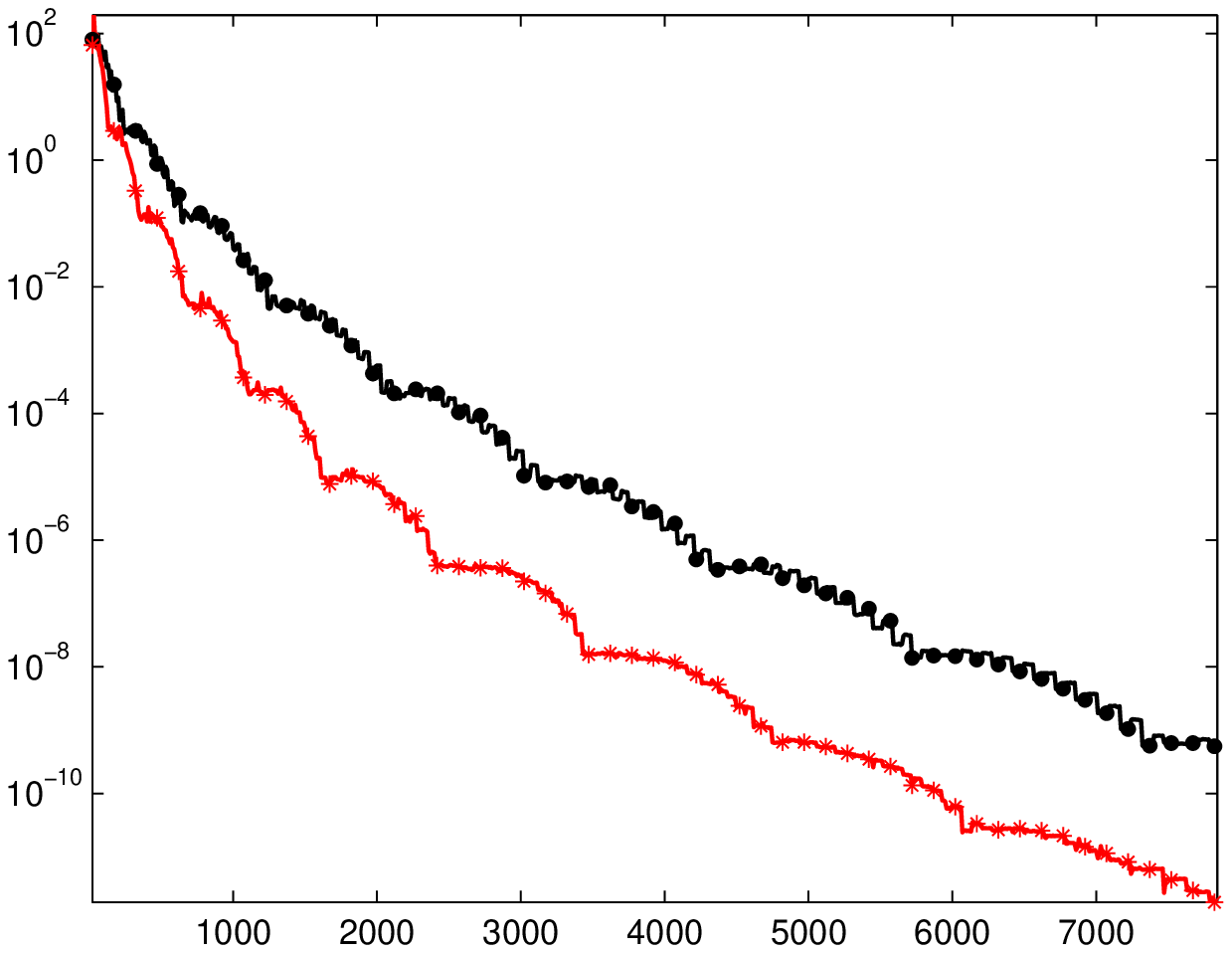}
\caption{Comparison between $E_{\hat{P}_{X}}(f_{7})$ (\textcolor[rgb]{0.98,0.00,0.00}{$*$}) (lower curve) and $E_{P_{X^{\prime }}}(f_{7})$ (\textcolor[rgb]{0.00,0.00,0.00}{$\bullet$}) (upper curve) for $20\leq n\leq7843$. When $n=7843$, $\deg(\hat{P}_{X}f_{7})=m+p^{*}=196+80$.}\label{caso39}
\end{minipage}
\end{figure}
Figure \ref{runge} compares the errors for $f_{2}$. The error in the constrained mock-Chebyshev least-squares is, for every $30\leq n\leq 3530$, smaller than the error in the mock-Chebyshev interpolation. The number $n=3530$ is due to the fact that the constrained mock-Chebyshev least-squares method reaches order $10^{-15}$ on $n+1=3531$ equispaced nodes. The accuracy of the mock-Chebyshev interpolation on the same set of nodes is of order $10^{-12}$.
Figure \ref{caso32} shows how the errors vary for the function $f_{5}$ when $20\leq n\leq 292$. Also in this case the approximation provided by the constrained mock-Chebyshev least-squares is more accurate than the one provided by the mock-Chebyshev interpolation and again when the accuracy of the former is of order $10^{-15}$ the accuracy of the latter is of order $10^{-11}$. Figure \ref{caso38} shows the errors behaviour for the function $f_{6}$ when $40\leq n\leq 924$ and the results are similar than in the previous cases. Finally, Figure \ref{caso39} compares the errors for $f_{7}$. In this case, the maximum order of precision that can  be reached by the constrained mock-Chebyshev method is $10^{-12}$.

\bigskip

The remaining part of the present Section is devoted to the comparison of the constrained mock-Chebyshev method with some Radial Basis Functions, Hermite Function interpolation (cf. \cite{boyd2013hermite}) and Floater-Hormann barycentric interpolation. A difference between these techniques and the constrained mock-Chebyshev least-squares is the structure of the approximation. Indeed, only the constrained mock-Chebyshev least-squares is based on polynomials, while the other approximants belong to other classes of functions.

\bigskip
\textit{Constrained mock-Chebyshev method vs RBF interpolation}
\bigskip
\\
\indent Given $n$ points $\xi_{1},\ldots \xi _{n}$ in $[-1,1]$ (called centers) and the corresponding values $f_{i}$ of a given function $f$ on them, an RBF interpolant for $f$ takes the form
\begin{equation*}
S(t)=\sum\limits_{i=1}^{n}\lambda _{i}\phi (\left\vert t-\xi
_{i}\right\vert )
\end{equation*}%
where $\phi (r)$ is a function defined for $r\geq 0$. The $\lambda _{i}$ are determined, as usual, by imposing the interpolation conditions $S(\xi _{j})=f_{j},$ $j=1,...,n$. Popular choices for $\phi (r)$ are (cf. \cite{fasshauer2007meshfree}):

\begin{itemize}
\item $\phi (r)=\left\vert r\right\vert ^{2m+1}$, Monomials (MN),

\item $\phi (r)=(1-r)_{+}^{4}(1+4r)$, Wendland (W2),

\item $\phi (r)=\frac{1}{\sqrt{1+(\varepsilon r)^{2}}}$, Inverse Multiquadric (IMQ),

\item $\phi (r)=\exp (-(\varepsilon r)^{2})\,$, Gaussian (G),
\end{itemize}
$\varepsilon$ is known as \textit{shape parameter} since as $\varepsilon\rightarrow 0$ RBFs become flater, while $\varepsilon\rightarrow \infty$ makes the RBFs spiky.
The first two are parameter-free and piecewise smooth, while Inverse Multiquadrics and Gaussians are
infinitely smooth and depend on $\varepsilon$. Although we will numerically compare the constrained mock-Chebyshev
method with the RBF interpolants associated to every choice of $\phi $
listed above, from a theoretical point of view we focus our
attention on the Gaussian RBFs (GRBFs).
In \cite{discroll2002interpolation} it has been proved that, when $\varepsilon
\rightarrow 0$, smooth RBF interpolants converges on the polynomial interpolants on the
same nodes. This means that, in such a flat limit case, as the polynomial
interpolation also the RBF approximation on uniform grids suffers of the
Runge phenomenon. Furthermore, in \cite{platte2011how} the author showed that the
GRBFs on equally spaced nodes and fixed parameter diverge when interpolating functions that have poles in the Runge region of polynomial interpolation.
A way to avoid the Runge phenomenon when interpolating with GRBF is to vary
the shape parameter with $n$. Indeed, as suggested in \cite{Boyd2010six}, if we define $\alpha
=\varepsilon \frac{2}{n}$, for $\alpha =O\left( \frac{%
1}{\sqrt[4]{n}}\right) $ the Runge phenomenon disappears. Such a choice has
a drawback since, as $n\rightarrow \infty $, the condition number of the
interpolation matrix increases exponentially. Hence, the GRBFs can defeat
the Runge Phenomenon just as the constrained mock-Chebyshev least-squares, but being ill-conditioned they can
be used only on few nodes. Ill-conditioning, mainly due to the basis of translates, can be
reduced significantly by using stable bases, as discussed in \cite{demarchi2012anewstable}.
\begin{figure}[htb]
\centering
\begin{minipage}[l]{.40\textwidth}
\includegraphics[width=7cm, height=5cm]{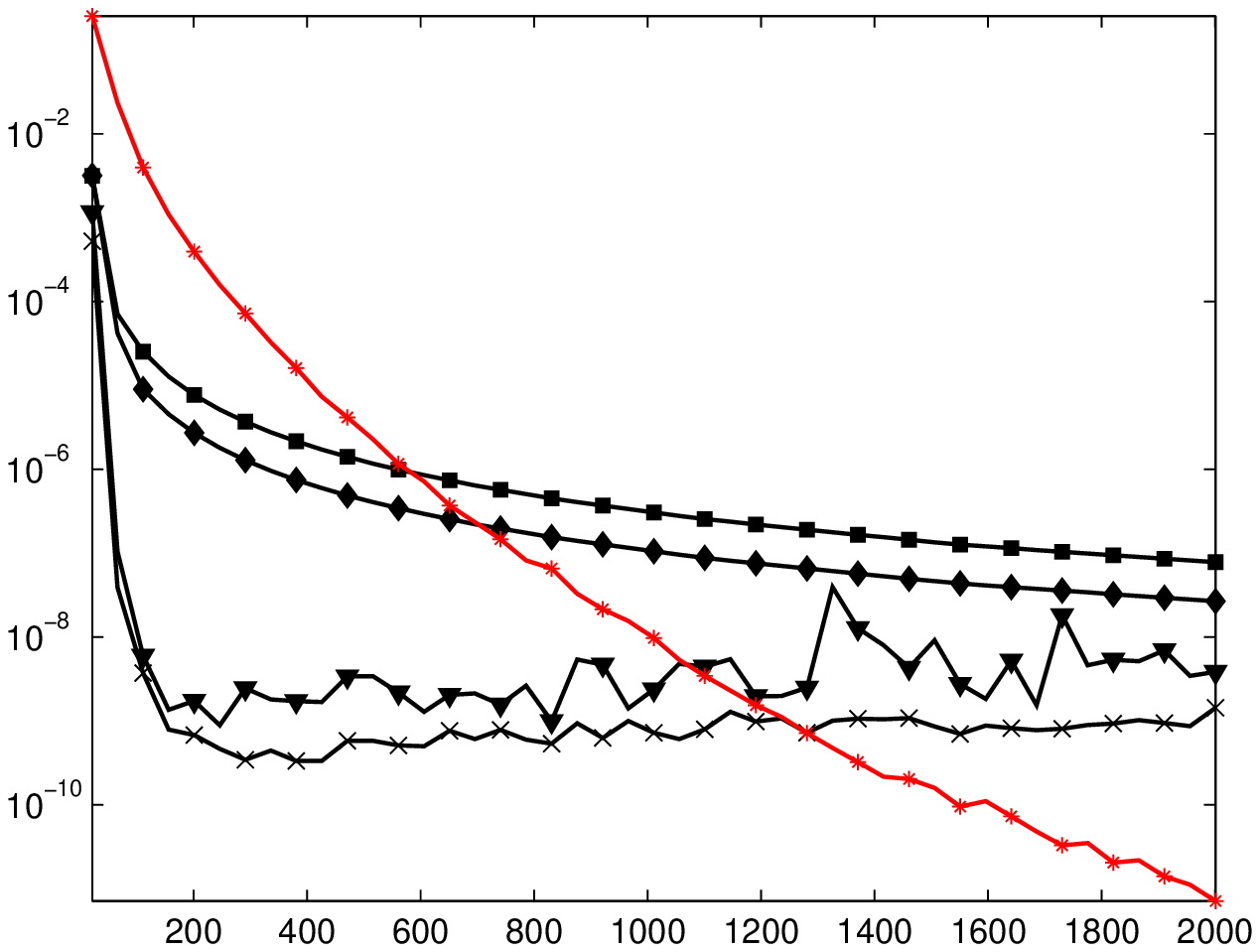}
\caption{Comparison between $E_{\hat{P}_{X}}(f_{2})$ (\textcolor[rgb]{0.98,0.00,0.00}{$*$}) and the errors obtained in approximating $f_{2}$ with (from top to bottom) W2 ($\sqbullet$), MN ($\blackdiamond$), G ($\blacktriangledown$), and IMQ ($\times$) RBF interpolants for $20\leq n\leq 2000$.}\label{rbf}
\end{minipage}\hspace{10mm}%
\begin{minipage}[l]{.40\textwidth}
\includegraphics[width=7cm, height=5cm]{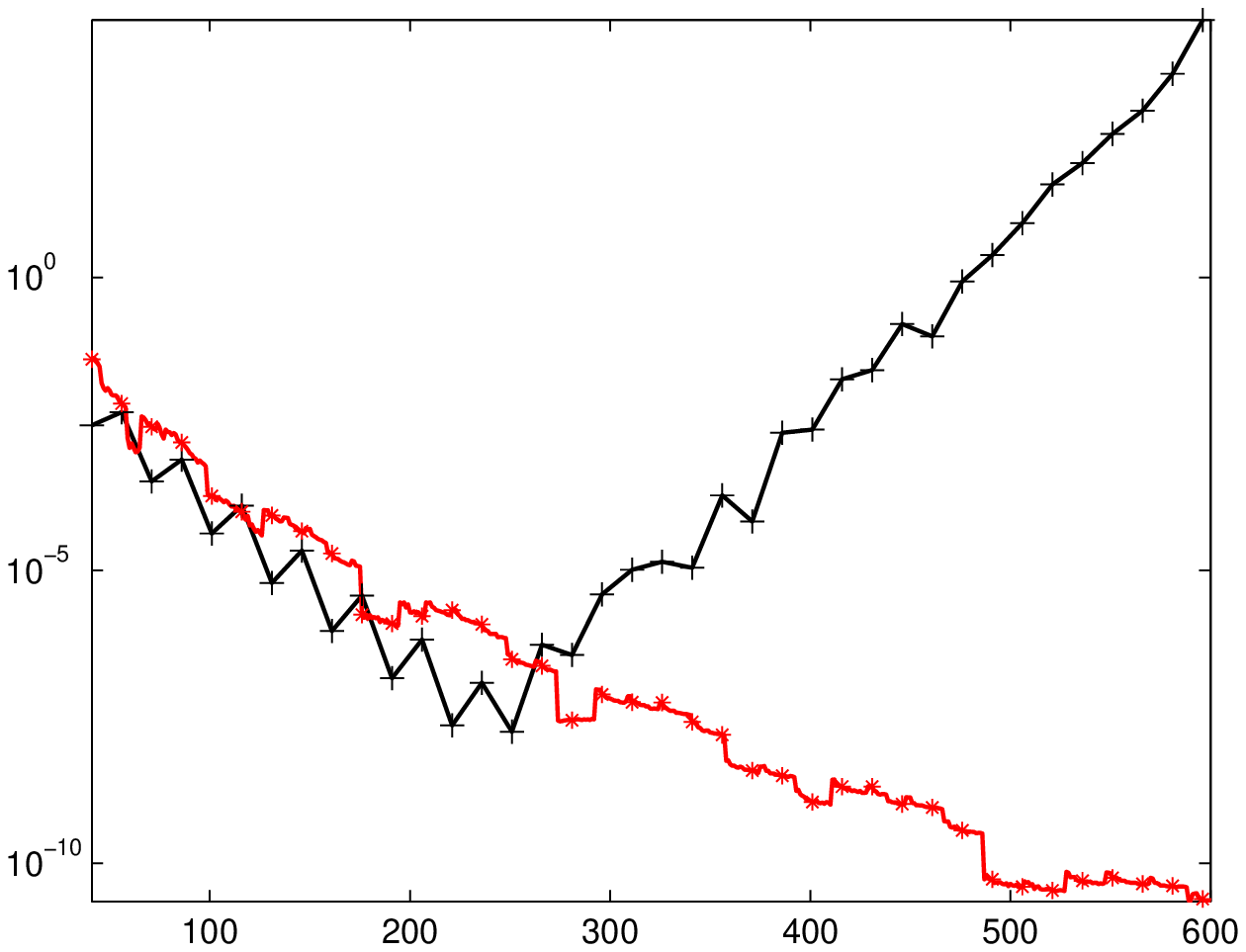}
\caption{Comparison between $E_{\hat{P}_{X}}(f_{6})$ (\textcolor[rgb]{0.98,0.00,0.00}{$*$}) and the error obtained in approximating $f_{6}$ with the Hermite function interpolant ($+$) for $40\leq n\leq 600$.}\label{hermite}
\end{minipage}
\begin{minipage}[l]{.40\textwidth}
\includegraphics[width=7cm, height=5cm]{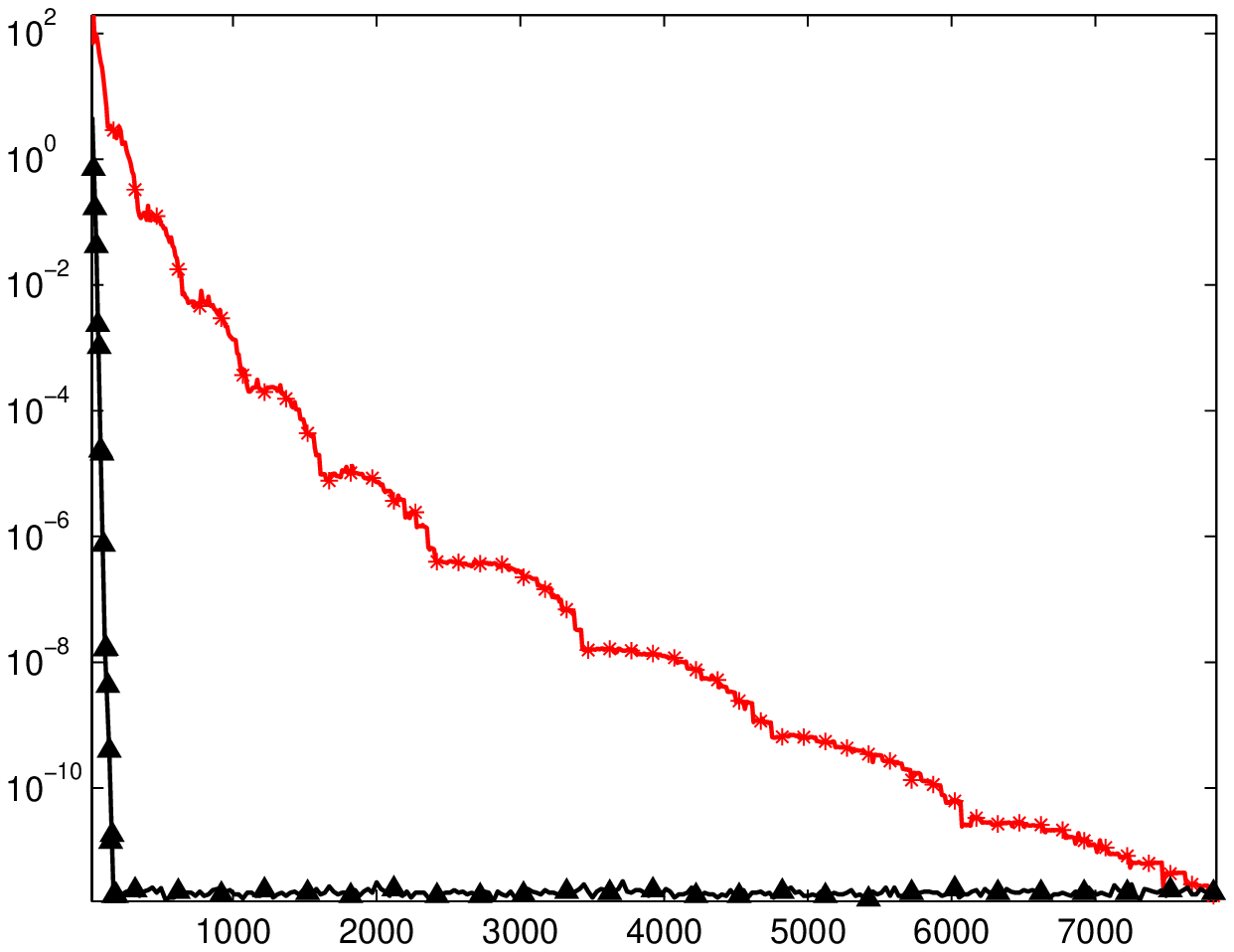}
\caption{Comparison between $E_{\hat{P}_{X}}(f_{7})$ (\textcolor[rgb]{0.98,0.00,0.00}{$ * $}) (upper curve), and the error in the Floater-Hormann barycentric interpolation (\textcolor[rgb]{0.00,0.00,0.00}{$\blacktriangleup$}) (lower curve) for $20\leq n\leq7843$.} \label{caso39mix}
\end{minipage}
\end{figure}

\indent Figure \ref{rbf} shows that, in approximating the Runge function $%
f_{2}$, the constrained mock-Chebyshev least-squares are, for initial values of $n$, less accurate than
the RBFs interpolants, while, as $n$
increases, they become more accurate. To have an idea of the discrepancy, while the constrained mock-Chebyshev
least-squares reach order $10^{-15}$ (see Figure \ref{runge}), the order of the RBFs interpolants for large $n$
ranges from $10^{-7}$ to $10^{-9}$. In performing this numerical test, for every fixed $n$, we have determined the shape parameter of IMQ and GRBFs using the so called \textit{Trial $\&$ Error} technique which consists in varying $\varepsilon$ into a fixed (discrete) range and choosing the $\lq \lq$optimal" parameter as the one that produces the minimum error. Unfortunately this method requires a lot of CPU time
for finding the $\lq \lq$optimal" shape parameter. Other techniques are also available, as
those described in \cite[Ch. 17]{fasshauer2007meshfree}, but for our purposes the Trial $\&$ Error was a suitable
way to estimate the optimal $\epsilon$.

\bigskip
\textit{Constrained mock-Chebyshev method vs Hermite function interpolation}
\bigskip
\\
\indent For a given function $f$ the Hermite function interpolant on $n$ points $\xi
_{1},\ldots \xi _{n}$ in $[-1,1]$ can be expressed
in the first barycentric form as%
\begin{equation*}
\begin{array}{ccc}
H(t)=\Omega (t)\sum\limits_{j=1}^{n}\frac{\mu _{j}}{t-\xi _{j}}f(\xi
_{j}), & \Omega(t)=\exp (-(n-1)/2\log (4)\gamma
^{2}t^{2})\prod\limits_{i=1}^{n}(t-\xi _{j}), & \mu _{j}=\left( \frac{%
d\Omega }{dt}(\xi _{j})\right) ^{-1}%
\end{array}
\end{equation*}
where $\gamma $ is a free parameter (optimal choices are $1$ or
slightly smaller). As stated in \cite{boyd2013hermite}, the computational cost
of the previous formula is $O(n^{2})$ which means that the Hermite function
interpolation is cheaper than the GRBF interpolation. Furthermore, in the same paper the authors
give numerical evidence that the Hermite function interpolation is substantially more accurate than the GRBF
interpolation. However, as RBFs, also this kind of interpolation is strongly
ill-conditioned and therefore its use must be limited to a maximum of about $250$
interpolation points. Figure \ref{hermite} shows how the ill-conditioning limits to $10^{-8}$ the best attainable accuracy in approximating $f_6$ with the Hermite interpolant, while the constrained mock-Chebyshev least-squares are very close to machine precision (see Figure \ref{caso38}).

\bigskip
\textit{Constrained mock-Chebyshev method vs Floater-Hormann interpolation}
\bigskip
\\
\indent A Floater-Hormann interpolant is a rational global approximant obtained blending local interpolating polynomials. More precisely, given $n+1$ distinct points $-1=x_0<
x_{2}<\ldots <x _{n}=1$ and fixed an integer $d$ such that $0\leq d\leq n$, a Floater-Hormann barycentric interpolant for $f$ can be written as
\begin{equation*}
R(t)=\sum\limits_{i=0}^{n-d}\nu_i(t)p_i(t)\Big{/}\sum\limits_{i=0}^{n-d}\nu_i(t)
\end{equation*}
where $p_i(t)$ is the polynomial of degree at most $d$ which interpolates $f$ in $x_i,\ldots,x_{i+d}$, $i=0,\ldots,n-d$, while
\begin{equation*}
\nu_i(t)=\frac{(-1)^i}{(t-x_i)\dots(t-x_{i+d})}.
\end{equation*}
This is a stable technique as confirmed by the study of the Lebesgue constant in \cite{bos2012onthelebesgue}. Looking at Figure \ref{caso39mix}, it is evident that, in approximanting $f_7$, the Floater-Hormann interpolant reaches $10^{-12}$ on few nodes, but then stabilizes without gaining anymore precision. Such a limit seems to be related to the smoothness of the function
and to the location of its poles within the Runge region. The error in the Floater-Hormann barycentric interpolation has been calculated using the Chebfun algorithms which for each value of $n$ choose the $\lq \lq$best" blending parameter \cite{chebfun}.

\bigskip

From previous comparisons we can conclude that the constrained mock-Chebyshev least-squares are a competitive \textit{polynomial strategy} for defeat the Runge phenomenon. In this context, we can affirm that this method currently provides the best we can expect from polynomials.

\section{Algorithm}

Let us recall that, fixed $p$ as in (\ref{p}), the polynomial $\hat{P}_{X}$ is given by%
\begin{equation*}
\hat{P}_{X}(t)=P_{X^{\prime }}(t)+\hat{Q}_{X^{\prime \prime }}(t)\omega
_{m}(t)
\end{equation*}%
where the polynomial $\hat{Q}_{X^{\prime \prime }}$ is the solution of the following least-squares problem
\begin{equation*}
\min_{Q\in \mathcal{P}^{p}}\left\Vert f-P_{X^{\prime }}-Q\omega
_{m}\right\Vert _{2}^{2}.
\end{equation*}%
We can express the previous minimum problem in matrix-form as follows
\begin{equation}
\min_{c\in
\mathbb{R}
^{p+1}}\left\Vert Ac-b\right\Vert _{2}^{2}  \label{minmat}
\end{equation}%
where $A=\left[ \omega _{m}(x_{i,n-m}^{\prime \prime
})\times(x_{i,n-m}^{\prime \prime })^{j-1}\right] _{\substack{ i=1,\ldots ,n-m \\ %
j=1,\ldots ,p+1}}$ is a real $(n-m)\times (p+1)$ matrix, $c=[c_1,\ldots,c_{p+1}]^{T}$ is the vector of coefficients of $\hat{Q}_{X^{\prime \prime }}$ and $b=\left[ P_{X^{\prime }}(x_{1,n-m}^{\prime \prime
})-f(x_{1,n-m}^{\prime \prime }),\ldots ,P_{X^{\prime }}(x_{n-m,n-m}^{\prime
\prime })-f(x_{n-m,n-m}^{\prime \prime })\right] ^{T}$.
Thus, the polynomial $\hat{P}_{X}$ can be computed using the following algorithm:
\begin{algorithm}[H]
\caption{Constrained mock-Chebyshev least-squares}
\label{alg}
\begin{algorithmic}
\REQUIRE $X_n$, the set of $n+1$ equispaced nodes in $[-1,1]$ and the evaluations of $f$ at $X_n$
\begin{enumerate}
\item Determine the subset $X'_m$ of $X_n$ whose elements are the nearest to the $m+1$ Chebyshev-Lobatto nodes and its complement $X''_{n-m}$;
\item Compute the polynomial $P_{X'}$ of degree $m$ which interpolates $f$ on $X'_m$;\label{pol}
\item Compute the polynomial $\omega_{m}$;\label{om}
\item Form the matrix $A$;
\item Solve $\min_{c\in\mathbb{R}^{p+1}}\left\Vert Ac-b\right\Vert _{2}^{2}$;\label{sist}
\end{enumerate}
\RETURN $\hat{P}_{X}=P_{X^{\prime }}+\hat{Q}_{X^{\prime \prime }}\omega_{m}.$
\end{algorithmic}
\end{algorithm}
For the sake of better readability, in Algorithm \ref{alg} we have not specified that, when we deal with the computation of a polynomial (cf. Steps \ref{pol}-\ref{om}), we refer to its evaluations on a given array. To improve the performance of this algorithm we implemented Step \ref{pol} using the barycentric formula (cf. \cite{Berrut_barycentriclagrange}). Such a formula is stable (cf. \cite{higham2004thenumerical}) and its computational cost is $O(m^2)=O(n)$. The evaluations of $\hat{Q}_{X^{\prime \prime }}$ and $\omega_m$ are performed using the Horner algorithm. Let us observe that Step \ref{sist} is the most expensive one. Since $A$ has full rank, if we solve (\ref{minmat}) with the Householder QR factorization
(which is a stable method) we need $2(n-m)(p+1)^{2}-2(p+1)^{3}/3$ flops (cf. \cite{golub1996matrix}).
Recalling that both $m$ and $p$ are proportional to $\sqrt{n}$, solving (\ref
{minmat}) requires $O(n^{2})$ flops. Thus, the cost of the constrained mock-Chebyshev least-squares is $O(n^{2})$.

\section{Conclusion and perspective}
In this work, we have combined the mock-Chebyshev interpolation with a simultaneous regression, to defeat the Runge Phenomenon for analytic functions with singularities close to the interval $[-1,1]$. We have determined a degree for the simultaneous regression and a sufficient condition under which for such a degree the error of the constrained mock-Chebyshev method is, in the uniform norm, less than the error of the mock-Chebyshev interpolation. The proposed examples confirms that, in the uniform norm, the
constrained mock-Chebyshev least-squares has better accuracy than the mock-Chebyshev interpolation. It might be interesting to extend this idea to the multivariate case on domains whose optimal distribution of nodes is known (cf. \cite{bos2007bivariate}).
\section*{Acknowledgements}
This work is supported by the "ex-$60\%$" funds of the University of Padova and by the project PRAT2012 of the University of Padova "Multivariate approximation with application to image reconstruction". We appreciated the
reviewers comments and suggestions that made the final version of the
paper more readable and clear. Furthermore, the authors would like to thank Prof. Marco Vianello of the University
of Padua for fruitful discussions with him. Finally, special thanks go to Prof. Stefano Serra-Capizzano of the University of Insubria for his valuable comments.





\bibliographystyle{ieeetr}
\bibliography{mariarosa}






\end{document}